\documentclass{article}
\pdfoutput=1

\usepackage[utf8]{inputenc}
\usepackage{amsmath,amssymb,amsthm}
\usepackage{mathtools}
\usepackage{microtype}
\usepackage[dvipsnames,x11names]{xcolor}
\usepackage[a4paper,bottom=4.5cm]{geometry}
\usepackage{caption}
\usepackage{subcaption}
\usepackage{algorithm2e}
\RestyleAlgo{ruled}
\DontPrintSemicolon
\usepackage{sectsty}
\allsectionsfont{\sffamily\mdseries\upshape} 
\usepackage[colorlinks = true,
            linkcolor = Blue,
            urlcolor  = Blue,
            citecolor = Blue,
            anchorcolor = black]{hyperref}            
            
\theoremstyle{plain}
\newtheorem{theorem}{Theorem}[section]
\newtheorem{lemma}[theorem]{Lemma}
\newtheorem{proposition}[theorem]{Proposition}

\theoremstyle{definition}
\newtheorem{remark}[theorem]{Remark}

\numberwithin{equation}{section}

\newcommand{\R}{{\mathbb R}}
\newcommand{\abs}[1]{{ \left \lvert #1 \right\rvert}}
\DeclareMathOperator{\argmin}{argmin}
\DeclareMathOperator{\rank}{rank}

\title{Local convergence of alternating low-rank optimization methods with overrelaxation}
\author{Ivan V. Oseledets\thanks{Skolkovo Institute of Science and Technology, Moscow, 121205 Russia} \and Maxim V. Rakhuba\thanks{HSE University, Moscow, 109028 Russia} \and Andr\'e Uschmajew\thanks{Max Planck Institute for Mathematics in the Sciences, 04103 Leipzig, Germany}
}
\date{}

\begin{document}

\maketitle

\begin{abstract}
The local convergence of alternating optimization methods with overrelaxation for low-rank matrix and tensor problems is established. The analysis is based on the linearization of the method which takes the form of an SOR iteration for a positive semidefinite Hessian and can be studied in the corresponding quotient geometry of equivalent low-rank representations. In the matrix case, the optimal relaxation parameter for accelerating the local convergence can be determined from the convergence rate of the standard method. This result relies on a version of Young's SOR theorem for positive semidefinite $2 \times 2$ block systems.
\end{abstract}

\section{Introduction}

We consider a low-rank matrix optimization problem of the form
\begin{equation}\label{eq: low rank optimization}
\min_{\rank (X) \le k} f(X),
\end{equation}
where $f : \R^{m \times n} \to \R$ is a smooth function on the space of real $m \times n$ matrices. It will be mostly assumed that $f$ is strongly convex. This generic problem appears in a large number of applications, where low-rank matrices serve as nonlinear model classes, such as in matrix recovery, or are employed for reducing numerical complexity when dealing \mbox{with large-scale matrices.}

Since the constraint set admits the explicit parametrization $X = UV^\top$, the problem can be rewritten as
\begin{equation}\label{eq: factorized optimization problem}
\min_{U \in \R^{m \times k}, \, V \in \R^{n \times k}} F(U,V) = f(UV^\top).
\end{equation}
One of the basic methods for solving~\eqref{eq: factorized optimization problem} is the alternating optimization (AO) method, which optimizes the factor matrices $U$ and $V$ in an alternating manner. Conceptually, ignoring the question of unique solvability of subproblems, the method looks as follows:
\begin{equation}\label{eq: block AO}
\begin{aligned}
U_{\ell + 1} &= \argmin_{U} F(U, V_\ell),  \\
V_{\ell + 1} &= \argmin_V F(U_{\ell+1}, V).
\end{aligned}
\end{equation}
While this is certainly a standard approach from the viewpoint of nonlinear optimization, where such a scheme is also known as nonlinear Gauss-Seidel method, it is worth emphasizing that the special structure of low-rank problems is particularly amenable to it. This is due to the bilinearity of the parametrization $U V^\top$, which turns the subproblems of~\eqref{eq: block AO} into optimization problems for the same initial function $f$, but on lower dimensional linear subspaces. Therefore, when $f$ is a quadratic function, this method is called the alternating least squares (ALS) method. 

While the study of global convergence of the AO method~\eqref{eq: block AO} is usually difficult, its local convergence properties are well-understood~\cite{Uschmajew2012,Rohwedder2013,Oseledets2018}. The local analysis is based on the fact that the linearized version of the method at a critical point $(U_*,V_*)$ takes the form of a block Gauss-Seidel method for the Hessian $\nabla^2 F(U_*,V_*)$. Due to the intrinsic overparametrization of rank-$k$ matrices by the representation $U V^\top$, the Hessian is at best positive semidefinite, but never positive definite. The Gauss-Seidel error iteration matrix is not contractive on the null space of the Hessian, but it turns out that this problem can be overcome by passing to the corresponding quotient geometry of equivalent low-rank representations $X = U V^\top$. This is possible thanks to an invariance of the AO method under changes of the representation. In fact, this invariance allows one to regard the method~\eqref{eq: block AO} as a well-defined iteration on the set of rank-$k$ matrices.

In this work, we consider the acceleration of the local convergence of the AO method~\eqref{eq: block AO} by means of overrelaxation. This is a classic idea in nonlinear optimization; see, e.g.,~\cite{Schechter1962, Ortega1966, Ortega1970, Hageman1975} to mention some early works. Several variants of such acceleration have been proposed for low-rank matrix problems, e.g., for matrix completion~\cite{Wen2012,Wang2019}. The basic overrelaxation method that we consider has already been proposed in~\cite{Grasedyck2015} for the more general low-rank tensor train format and in the matrix case reads as follows:
\begin{equation}\label{eq: block AO shift}
\begin{aligned}
U_{\ell + 1} &= (1-\omega) U_\ell + \omega \argmin_{U} F(U, V_\ell),  \\
V_{\ell + 1} &= (1-\omega) V_\ell + \omega \argmin_V F(U_{\ell+1}, V).
\end{aligned}
\end{equation}
Here $\omega > 0$ is a relaxation parameter, which sometimes is also called a shift. It can be observed numerically that a suitable choice of the shift significantly improves the convergence speed.

Our goal is to study the local convergence of this iteration for low-rank optimization in a similar spirit as for the plain AO method~\eqref{eq: block AO}, which corresponds to the case $\omega = 1$. This will be done in the main section~\ref{sec: low-rank AO}. The linearization of~\eqref{eq: block AO shift} at a critical point $(U_*,V_*)$ leads to a $2 \times 2$ block SOR method for the Hessian $\nabla^2 F(U_*, V_*)$. Using the fact that for $0 < \omega < 2$ and a positive semidefinite Hessian with positive definite block diagonal the SOR error iteration is contractive on any subspace complementary to the null space of the Hessian, we obtain local convergence results for this range of $\omega$. This result is stated in Theorem~\ref{thm: local convergence theorem}.

It is then natural to ask for the optimal shift $\omega$ achieving the fastest local convergence rate, which requires to minimize the spectral radius of the SOR error iteration matrix. For positive definite $2 \times 2$ block systems this can be achieved using a well-known theorem of Young. It is however possible to adjust the arguments to the positive semidefinite case, as will be done in Lemma~\ref{lem: SOR method}. This yields the expected, yet not entirely trivial, formula for the asymptotically optimal shift in terms of the convergence rate of the standard AO method with $\omega = 1$. The result is stated in Theorem~\ref{thm: optimal shift for AO}. In practice this means that the optimal shift can be estimated adaptively and at practically zero cost from the observed convergence rate of the standard method.

Of course, overrelaxation can also be applied to AO methods for low-rank tensor optimization. In section~\ref{sec: low-rank tensor problems} we focus on the low-rank tensor train format as in~\cite{Grasedyck2015}. Like low-rank matrix factorization, the tensor train decomposition is subject to an intrinsic overparametrization which can be described by a simple group action in parameter space, but leads to formally semidefinite Hessians in critical points. By passing to suitable quotient spaces, the local convergence of the method can be established in essentially the same way as for low-rank matrices (Theorem~\ref{thm: local convergence TT}). However, a main difference to the matrix case is that the formula for the optimal shift cannot be made rigorous under reasonable assumptions, although it can still serve as a useful heuristic.

In the last section~\ref{sec: numerical experiments} we report on some numerical experiments that illustrate the advantage of using shifts in low-rank AO methods, and validate our theoretical findings regarding the optimal shift in the matrix case. We also demonstrate the adaptive procedure for choosing an almost optimal shift based on the observed convergence rate of the standard method.

\section{Alternating optimization with relaxation for low-rank matrices}\label{sec: low-rank AO}

In this main section of the paper we first formalize the basic AO iteration~\eqref{eq: block AO} for low-rank matrix problems and recall some of its basic properties. We then proceed to the method with overrelaxation, establish its local convergence and determine the optimal shift parameter.

\subsection{Standard AO method}

Consider the scheme~\eqref{eq: block AO} and assume $f$ to be strongly convex. Then the first argmin is uniquely defined if $\rank(V_\ell) = k$, since it corresponds to minimizing the strongly convex function $f$ on a linear subspace of $\R^{m \times n}$ which is the image of the injective linear map $U \mapsto U V_\ell^\top$. Likewise, the second argmin is well-defined if $\rank(U_{\ell+1}) = k$ and returns the unique minimum of $f$ on the linear subspace $V \mapsto U_{\ell+1} V^\top$. Therefore, in some open and dense subsets both argmins define smooth maps $\hat S_1$ and $\hat S_2$, respectively, such that
\begin{equation}\label{eq: update AO}
U_{\ell + 1} = \hat S_1(V_\ell), \quad V_{\ell+1} = \hat S_2(U_{\ell+1}).
\end{equation}
One full update of the method then takes the form of a fixed point iteration
\begin{equation}\label{eq: AO as fpi}
\begin{pmatrix}
U_{\ell+1} \\ V_{\ell+1} 
\end{pmatrix}
= S \begin{pmatrix} U_\ell \\ V_\ell \end{pmatrix} \coloneqq \begin{pmatrix} \hat S_1(V_\ell) \\ \hat S_2(\hat S_1(V_\ell))
\end{pmatrix}.
\end{equation}
The map $S$ is well-defined and smooth on any open subset of
\[
\mathcal D = \{ (U,V) \in \R^{m \times k} \times \R^{n\times k} \colon \text{$V$ and $\hat S_1(V)$ have full column rank $k$} \}.
\]
In particular, any critical point $(U_*,V_*)$ of $F$ in~\eqref{eq: factorized optimization problem} for which $U_*$ and $V_*$ have full column rank belongs to $\mathcal D$ and is a fixed point of $S$. To see this, note that $U \mapsto F(U,V_*) = f(UV_*^\top)$ is strongly convex since $U \mapsto U V_*^\top$ is injective. Since $U_*$ is a critical point of that function, it is the global minimum and hence $\hat S_1(V_*) = U_*$. The argument for $V_*$ is the same. Such a critical point of $F$ possesses an open neighborhood in $\mathcal D$ in which $S$ is well-defined and smooth. Conversely, any fixed point $(U_*,V_*) \in \mathcal D$ of $S$ must be a critical point of $F$ since it implies that $U_*$ is the global minimum of $U \mapsto F(U,V_*)$ and $V_*$ is the global minimum of $V \mapsto F(U_*,V)$. Hence the partial gradients $\nabla_U F(U_*,V_*)$ and $\nabla_V F(U_*,V_*)$ are both zero.

By passing from the initial constrained problem~\eqref{eq: low rank optimization} to the factorized problem~\eqref{eq: factorized optimization problem}, we formally introduced an ambiguity arising from the fact that the factorization $X = UV^\top$ of a rank-$k$ matrix is not unique. In particular, $X = UAA^{-1}V^\top$ for any invertible $k \times k$ matrix $A$ so that the function $F$ has level sets of at least dimension $k^2$ (when $U,V$ have full column rank). Therefore, a fixed point $(U_*,V_*) \in \mathcal D$ of $S$ is never locally unique. However, this issue of non-uniqueness is only a formal one since one is ultimately interested in the sequence of generated matrices $X_\ell^{} = U_\ell^{} V_\ell^\top$. Assuming $\rank(X_\ell) = k$ for all $\ell$, this sequence is not affected by any reparametrization $(U_\ell^{},V_\ell^{}) \to (U_\ell^{} A_\ell^{}, V_\ell^{} A^{-T}_\ell)$ with invertible matrices $A_\ell$ during the iteration. This is due to the invariance properties
\begin{equation}\label{eq: S1 S2 invariance}
\begin{aligned}
\hat S_1(VA^{-T}) &= \hat S_1(V) A, \\ 
\hat S_2(UA) &= \hat S_2(U) A^{-T}
\end{aligned}
\end{equation}
of the maps $\hat S_1$ and $\hat S_2$, which hold whenever $U$ and $V$ have full column rank. To see this, let $U_+ = \hat S_1(V)$ and $\hat U_+ = \hat S_1(VA^{-T})$. Then by construction $U_+ V^\top$ and $\hat U_+ A^{-1} V^\top$ are the unique minimizers of $f$ on the linear subspaces $\{ U V^\top \colon U \in \R^{m \times k}\}$ and $\{ U A^{-1}V^\top \colon U \in \R^{m \times k}\}$, respectively. Obviously both spaces are equal, hence $U_+ V^\top = \hat U_+ A^{-1} V^\top$. Since $V$ has full column rank, we obtain $U_+ = \hat U_+ A^{-1}$, the first identity in~\eqref{eq: S1 S2 invariance}. The argument for $\hat S_2$ is analogous. 

The above invariance of the AO method allows us to interpret it as a method
\[
X_{\ell+1} = \mathbf{S}(X_\ell)
\]
in the full matrix space $\R^{m \times n}$, or more precisely on the subset of matrices of rank at most~$k$. This viewpoint has been taken in~\cite{Oseledets2018} and will be helpful in this work, too. From an algorithmic perspective, the AO viewpoint~\eqref{eq: AO as fpi} is more useful since it operates on the smaller matrices $U$ and~$V$ instead of the full matrix $X$. Furthermore, the invariance with respect to the described change of parametrization allows for a robust implementation of the AO method by orthogonalizing the columns of $U_\ell$ and $V_\ell$ after every partial update, without affecting the generated sequence $X_\ell$ of matrices. This method is a special case of Algorithm~\ref{alg:als_qr} below with $\omega = 1$.

\subsection{Overrelaxation}

Instead of~\eqref{eq: update AO} we now consider the more general update rule with a shift,
\begin{equation}\label{eq:AO with relaxation}
\begin{aligned}
U_{\ell+1} &= (1-\omega) U_\ell + \omega \hat S_1(V_\ell),\\
V_{\ell+1} &= (1-\omega) V_\ell + \omega \hat S_2(U_{\ell+1}),
\end{aligned}
\end{equation}
which corresponds to~\eqref{eq: block AO shift}. For $\omega = 1$, this iteration equals the standard AO method~\eqref{eq: update AO}. By defining the map
\begin{equation}\label{eq: map Somega}
S_\omega \begin{pmatrix} U \\ V \end{pmatrix} = (1-\omega) \begin{pmatrix} U \\ V \end{pmatrix} + \omega \begin{pmatrix} \hat S_1(V) \\ \hat S_2((1-\omega) U + \omega \hat S_1(V)) \end{pmatrix},
\end{equation}
we can write~\eqref{eq:AO with relaxation} as a nonlinear fixed point iteration
\begin{equation}\label{eq: overrelaxation as fpi}
\begin{pmatrix} U_{\ell+1} \\ V_{\ell+1} \end{pmatrix} = S_\omega \begin{pmatrix} U_\ell \\ V_\ell \end{pmatrix}.
\end{equation}
The map $S_\omega$ is well-defined and smooth on any  open subset of 
\[
\mathcal D_\omega = \{ (U,V) \in \R^{m \times k} \times \R^{n\times k} \colon \text{$V$ and $(1-\omega)U + \omega \hat S_1(V)$ have full column rank} \}.
\]
Note that $(U_*,V_*) \in \mathcal D_\omega$ is a fixed point of $S_\omega$ if and only if $(U_*,V_*) \in \mathcal D$ and $(U_*,V_*)$ is a fixed point of $S$. In particular, any critical point $(U_*,V_*)$ of $F$ in~\eqref{eq: factorized optimization problem} such that $U_*$ and $V_*$ have full column rank belongs to $\mathcal D_\omega$ and is a fixed point of $S_\omega$. Moreover, any such critical point possesses an open neighborhood in $\mathcal D_\omega$ such that $S_\omega$ is well-defined and smooth on this neighborhood. Again, the converse is also true, that is, a fixed point $(U_*,V_*) \in \mathcal D_{\omega}$ of $S_\omega$ is a critical point of $F$.

The iteration~\eqref{eq: overrelaxation as fpi} exhibits the same invariance under changes of representation as the standard AO method. Let $(U,V) \in \mathcal D_\omega$ and $(U_+,V_+) = S_\omega(U,V)$, then from~\eqref{eq: map Somega} and~\eqref{eq: S1 S2 invariance} one verifies
\begin{equation}\label{eq: invariance}
S_\omega \begin{pmatrix} UA \\ VA^{-T} \end{pmatrix} = \begin{pmatrix} U_+ A \\ V_+ A^{-T} \end{pmatrix}.
\end{equation}
Therefore, the generated sequence $X_\ell^{} = U_\ell^{} V_\ell^\top$ is essentially (if $\rank(X_\ell) = k$ for all $\ell$) invariant under changes of representation during the iteration. In particular, QR decomposition can be used in numerical implementation for keeping the argmin problems well-conditioned. The resulting method is denoted in Algorithm~\ref{alg:als_qr}.

\begin{algorithm}
\caption{Low-rank AO with overrelaxation and QR}\label{alg:als_qr}
\KwData{$V_0\in\mathbb{R}^{n\times k}$, relaxation parameter $\omega$}
\For{$\ell = 0, 1, 2, \dots$}{
   $U \leftarrow \argmin_{\hat U \in\mathbb{R}^{m\times k}} F(\hat U, V_\ell)$\;
   $U \leftarrow (1-\omega) U_\ell + \omega U$, \quad $U = Q_1 R_1$\;
   $V \leftarrow \argmin_{\hat V\in \mathbb{R}^{n\times k}} F(Q_1, \hat V)$\;
   $V \leftarrow (1-\omega) V_\ell^{} R_1^\top + \omega V$, \quad $V = Q_2 R_2$\;
   $U_{\ell+1} \coloneqq Q_1^{} R_2^\top$, \quad $V_{\ell+1} \coloneqq Q_2$\;
}
\end{algorithm}

Since the goal of this work is a local convergence analysis of the fixed point iteration~\eqref{eq: overrelaxation as fpi}, it is important to observe that the invariance under the group action also carries over to the asymptotic linear convergence rate of the method, which depends on the eigenvalues of the derivative $S_\omega'(U_*,V_*)$ at a fixed point $(U_*,V_*) \in \mathcal D_
\omega$. To see this invariance it is convenient to introduce the corresponding group action $\theta_A$ of $\mathrm{GL}(k)$ acting on $\R^{m \times k} \times \R^{n \times k}$ via
\begin{equation}\label{eq: group action}
A \mapsto \theta_A \cdot \begin{pmatrix} U \\ V \end{pmatrix} = \begin{pmatrix} UA \\ VA^{-T} \end{pmatrix}.
\end{equation}
In this notation~\eqref{eq: invariance} reads
\begin{equation}\label{eq: invariance for Somega}
S_\omega \left(\theta_A \cdot \begin{pmatrix} U\\ V \end{pmatrix} \right) = \theta_A \cdot S_\omega \begin{pmatrix} U \\ V \end{pmatrix}.
\end{equation}
For fixed $A$,~\eqref{eq: group action} defines an invertible linear map $\theta_A$ on $\R^{m \times k} \times \R^{n \times k}$ with $\theta_A^{-1} = \theta_{A^{-1}}^{}$. Differentiating both sides of~\eqref{eq: invariance for Somega} it then follows that
\begin{equation}\label{eq: similarity of derivatives}
S_\omega' \begin{pmatrix} U_* A \\ V_* A^{-T} \end{pmatrix} = \theta_A \cdot S_\omega'\begin{pmatrix} U_* \\ V_* \end{pmatrix} \cdot \theta_{A}^{-1}.
\end{equation}
This shows that $S_\omega' ( U_* A, V_* A^{-T} )$ has the same eigenvalues as $S_\omega'( U_*, V_* )$ and allows us to study the local convergence rate of the iteration~\eqref{eq: overrelaxation as fpi} at any particular fixed point $(U_*, V_*)$.

As for the standard AO method, the invariance property allows for an interpretation of the method~\eqref{eq: overrelaxation as fpi} as an iteration
\begin{equation}\label{eq: iteration in full space}
X_{\ell + 1} = \mathbf{S}_\omega(X_\ell)
\end{equation}
on the manifold
\[
\mathcal M_k = \{ X \in \R^{m \times n} \colon \rank(X) = k \},
\]
where $\mathbf S_\omega \colon \mathcal O \subseteq \mathcal M_k \to \R^{m \times n}$ is defined through
\begin{equation}\label{eq: bold S omega}
\mathbf S_\omega(X) = \tau (S_\omega(U,V)), \quad X = U V^\top,
\end{equation}
with the map
\[
\tau(U,V) = U V^\top.
\]
Here the domain of definition $\mathcal O$ of $\mathbf{S}_\omega$ should be contained in the image of $\mathcal D \cap \mathcal D_\omega$ under $\tau$. In particular, let $(U_*,V_*) \in \mathcal D$ be a fixed point of the map $S$ (the standard AO method), that is, a critical point of $F$. Then $\mathbf{S}_\omega$ is well-defined and smooth in some open neighborhood $\mathcal O \subseteq \mathcal M_k$ of $X_*^{} = U_*^{} V_*^\top$ and $X_*$ is a fixed point of $\mathbf{S}_\omega$. This manifold viewpoint will be useful in the local convergence analysis conducted in the next section.

\subsection{Local convergence}

Let $X_*^{} = U_*^{} V_*^\top \in \mathcal M_k$ be a fixed point of $\mathbf S_\omega$. Then $\mathbf S_\omega$ locally maps to $\mathcal M_k$ and thus the derivative $\mathbf S'(X_*)$ maps the tangent space $T_{X_*} \mathcal M_k$ to itself. This provides the following local convergence criterion.

\begin{proposition}\label{prop: convergence criterion}
Let $f$ be strongly convex and $(U_*,V_*)$ be a critical point of $F$ in~\eqref{eq: factorized optimization problem} with $U_*, V_*$ having full column rank $k$. Then $(U_*,V_*)$ is a fixed point of $S_\omega$ and $X_*^{} = U_*^{} V_*^\top \in \mathcal M_k$ is a fixed point of $\mathbf{S}_\omega$. Let $\mathbf S_\omega' (X^*) \colon T_{X_*} \mathcal M_k \to \R^{m \times n}$ denote the derivative of $\mathbf S_\omega$ at $X_*$, and $\mathbf P_{X_*} \colon \R^{m \times n} \to T_{X_*} \mathcal M_k$ the tangent space projection. If for the spectral radius
\begin{equation}\label{eq: convergence criterion}
\rho_\omega = \rho( \mathbf P_{X_*} \mathbf S_\omega' (X_*)) < 1,
\end{equation}
then for $X_0^{} = U_0^{} V_0^\top$ close enough to $X_*$ the iterates $X_\ell^{} = U_\ell^{} V_\ell^\top$ generated by Algorithm~\ref{alg:als_qr} converge to $X_*$ at an asymptotic linear rate $\rho_\omega$.
\end{proposition}

To study the convergence criterion in more detail we investigate $\mathbf S_\omega' (X_*)$. For this we repeat some well-known computations. We first consider the map $S_\omega$ in parameter space. From~\eqref{eq: map Somega}, its derivative at $(U_*,V_*)$ takes the form of a block matrix
\[
S_\omega' \begin{pmatrix} U_* \\ V_* \end{pmatrix} = (1-\omega) \begin{pmatrix} I & 0 \\ 0 & I \end{pmatrix} + \omega \begin{pmatrix} 0 & \hat S_1'(V_*) \\ (1-\omega) \hat S_2'(U_*) & \omega \hat S_2'(U_*) \hat S_1'(V_*) \end{pmatrix},
\]
where we have used that $(1-\omega) U_* + \omega \hat S_1(V_*) = U_*$. Setting
\[
L = \begin{pmatrix} 0 & 0 \\ \hat S_2'(U_*) & 0\end{pmatrix}, \quad R = \begin{pmatrix} 0 & \hat S_1'(V_*) \\ 0 & 0 \end{pmatrix}
\]
one then verifies the identity
\begin{equation}\label{eq: DSomega}
S_\omega'\begin{pmatrix} U_* \\ V_* \end{pmatrix} = (I - \omega L)^{-1}[ (1-\omega)I + \omega R].
\end{equation}
This linear operator can be interpreted as a block SOR error iteration matrix for the Hessian
\[
H = \begin{pmatrix} \nabla_{UU}F(U_*,V_*) & \nabla_{UV}F(U_*,V_*) \\ \nabla_{VU}F(U_*,V_*) & \nabla_{VV}F(U_*,V_*) \end{pmatrix}
\]
of $F$ at $(U_*,V_*)$, written in $2 \times 2$ block form. To see this, consider the usual decomposition
\[
H = D + E + E^\top
\]
with
\[
D = \begin{pmatrix} \nabla_{UU}F(U_*,V_*) & 0 \\ 0 & \nabla_{VV}F(U_*,V_*) \end{pmatrix},  \quad E = \begin{pmatrix} 0 & 0 \\ \nabla_{VU}F(U_*,V_*) & 0  \end{pmatrix}.
\]
Assuming that the block diagonal part $D$ is invertible and differentiating the equations
\begin{equation}\label{eq: implicit definition S1 S2}
\nabla_U F(\hat S_1(V),V) = 0, \quad \nabla_V F(U, \hat S_2(U)) = 0
\end{equation}
(which implicitly define $\hat S_1$ and $\hat S_2$), one finds that
\[
\hat S_1'(V_*) = - [\nabla_{UU} F(U_*,V_*)]^{-1} \nabla_{UV} F(U_*,V_*)
\]
and
\[
\hat S_2'(U_*) = - [\nabla_{VV} F(U_*,V_*)]^{-1} \nabla_{VU} F(U_*,V_*).
\]
In other words, $L = -D^{-1}E$, $R = -D^{-1}E^\top$, and
\begin{equation*}\label{eq: L+R}
L + R = I - D^{-1} H.
\end{equation*}
Using the expressions for $L$ and $R$ in~\eqref{eq: DSomega}, one obtains the alternative formula
\begin{equation}\label{eq: derivative in parameter space}
S_\omega'\begin{pmatrix} U_* \\ V_* \end{pmatrix} = T_\omega \coloneqq I - N_\omega^{-1} H, \quad N_\omega = \frac{1}{\omega}D + E.
\end{equation}

We see from~\eqref{eq: derivative in parameter space} that $S_\omega'(U_*, V_*)$ equals the error iteration matrix $T_\omega$ for the two-block SOR method for $H$. It is well-known that $T_\omega$ has spectral radius less than one if $0 < \omega < 2$ and $H$ is positive definite. However, the latter is never the case here. Since $\nabla F$ is constantly zero on the orbit $\theta_A \cdot (U_*,V_*)$ (this follows from the chain rule by differentiating $F = F \circ \theta_A$ for fixed $A$), the Hessian at critical points $(U_*,V_*) \in \mathcal D$ has at least a $k^2$-dimensional kernel $\ker H$ containing the tangent space to the orbit. On $\ker H$ the matrix $T_\omega$ acts as identity. However, if $0 < \omega < 2$, $D$ is positive definite and $H$ is at least positive semidefinite, then by classic results it still holds that $T_\omega$ is a contraction on any invariant subspace complementary to $\ker H$; see, e.g.,~\cite[Sec.~3]{Weissinger1953} or~\cite[Corollary~2.1]{Keller1965}. Specifically, as follows from~\cite{Weissinger1953}, the space
\begin{equation}\label{eq: invariant subspace}
\mathcal W_\omega = N_\omega^{-1} (\ker H)^\perp
\end{equation}
is an invariant subspace of $T_\omega$, which splits the parameter space into a direct sum\footnote{Obviously, $\mathcal W_\omega$ is an invariant subspace of $T_\omega$ and has the correct dimension. To see that it is complementary to $\ker H$, note that any $x \in \mathcal W_\omega$ satisfies $N_\omega x \in (\ker H)^\perp$. If $x \in \ker H$, one verifies
 $0 = \langle x, N_\omega x \rangle = (\frac{1}{\omega} - \frac{1}{2}) \langle x, D x \rangle $, which under the given assumptions implies $x=0$.\label{footnote 1}}
\begin{equation}\label{eq: splitting}
\R^{m \times k} \times \R^{n \times k} = \ker H \oplus \mathcal W_\omega,
\end{equation}
and $T_\omega$ is a contraction on $\mathcal W_\omega$.

At this point we can exploit that we are actually interested in the convergence of the products $X_\ell^{} = U_\ell^{} V_\ell^\top$. Under the assumption that $\ker H$ equals the tangent space to the orbit $\theta_A \cdot ( U_*, V_* )$, any complementary subspace, such as $\mathcal W_\omega$, satisfies the properties of a so-called horizontal space for the quotient manifold structure of $\mathcal M_k$. For us, this means the following.

\begin{proposition}\label{prop: tau as diffeomorphism}
Assume $X_*^{} = U_*^{} V_*^\top$ has rank $k$, $H$ is positive semidefinite, $\dim (\ker H) = k^2$ and a decomposition~\eqref{eq: splitting} holds. Then the map $\tau(U,V) = UV^\top$ is a local diffeomorphism between a (relative) neighborhood of $(U_*,V_*)$ in $(U_*,V_*) + \mathcal W_\omega$ and a neighborhood of $X_*$ in the embedded submanifold $\mathcal M_k \subseteq \R^{m \times n}$.
\end{proposition}

\begin{proof}
The proof can be given without particular reference to quotient manifolds, but assuming knowledge that $\mathcal M_k$ is a smooth embedded submanifold of dimension $mk + nk - k^2$~\cite[Ex.~8.14]{Lee2003}, and $\tau$ is a local submersion on $\mathcal M_k$ in a neighborhood of $(U_*,V_*)$ (it is not difficult to verify that $\tau'(U_*,V_*)$ has rank $mk + nk - k^2$). Then since $\tau$ is constant on the $\theta_A$-orbit of $(U_*,V_*)$, its derivative vanishes on the tangent space to that orbit at $(U_*,V_*)$, which is of dimension $k^2$. We already noted that $\ker H$ contains that tangent space, so if $\dim (\ker H) = k^2$, then $\tau'(U_*,V_*)$ vanishes on $\ker H$. Hence, due to~\eqref{eq: splitting}, $\tau'(U_*,V_*)$ must be a bijection between $\mathcal W_\omega$ and the tangent space $T_{X_*} \mathcal M_k$. The assertion follows by the inverse function theorem.
\end{proof}

We are now in the position to formulate a local convergence result for the iteration~\eqref{eq: iteration in full space}. 

\begin{theorem}\label{thm: local convergence theorem}
Let $f$ be strongly convex and $(U_*,V_*)$ be a critical point of $F$ in~\eqref{eq: factorized optimization problem} with $U_*, V_*$ having full column rank $k$. Assume that the Hessian $H = \nabla^2F (U_*,V_*)$ is positive semidefinite and $\dim (\ker H) = k^2$. Fix $0 < \omega < 2$. Then for $X_0^{} = U_0^{} V_0^\top$ close enough (this may depend on~$\omega$) to $X_*^{} = U_*^{} V_*^\top$ Algorithm~\ref{alg:als_qr} is well-defined and the iterates $X_\ell^{} = U_\ell^{} V_\ell^\top$ converge to $X_*$ at an asymptotic linear rate
\[
\rho_\omega = \limsup_{\ell \to \infty} \| X_\ell - X_* \|^{1/\ell} < 1,
\]
where $\rho_\omega$ is the spectral radius of $T_\omega$ on $\mathcal W_\omega$.
\end{theorem}

The convergence rate $\rho_\omega$ is determined in the next section. We stated the above result separately because its proof can be easily generalized to alternating optimization methods for low-rank tensor formats that admit a similar invariance under a group action. This will be outlined for the tensor train format in section~\ref{sec: low-rank tensor problems}.

\begin{proof}
Since $V_*$ has full column rank, the linear map $U \mapsto U V_*^\top$ is injective and hence the restricted map $U \mapsto F(U, V_*) = f(U V_*^T)$ is strongly convex. Therefore, $\nabla_{UU}F(U_*,V_*)$ is positive definite. Likewise, $\nabla_{VV}F(U_*,V_*)$ is positive definite, so that the block diagonal part $D$ of $H$ is positive definite. As a result, the decomposition~\eqref{eq: splitting} of the parameter space applies and $T_\omega$ is a contraction on its invariant subspace $\mathcal W_\omega$. In a neighborhood of $X_*$ in $\mathcal M_k$ the map $\mathbf S_\omega$ in~\eqref{eq: bold S omega} can be written as
\[
\mathbf S_\omega = \tau \circ S_\omega \circ \tau^{-1},
\]
where we have restricted $\tau$ to the affine subspace $(U_*,V_*) + \mathcal W_\omega$. Therefore, by chain rule,
\begin{equation}\label{eq: dSomega with tau}
\mathbf S_\omega'(X_*) = [\tau'(U_*,V_*)] \circ T_\omega \circ [\tau'(U_*,V_*)]^{-1}.
\end{equation}
 By Proposition~\ref{prop: tau as diffeomorphism}, the derivative $\tau'(U_*,V_*)$ is an isomorphism between $\mathcal W_\omega$ and $T_{X_*} \mathcal M_k$.  Due to~\eqref{eq: dSomega with tau}, this implies that the convergence criterion~\eqref{eq: convergence criterion} in Proposition~\ref{prop: convergence criterion} is satisfied. 
\end{proof}

\begin{remark}\label{rem: D positive definite}
The assumptions that $H$ is positive semidefinite and $\dim(\ker H) = k^2$ already imply by themselves that $D$ is positive definite. Indeed, as noted in the proof of Proposition~\ref{prop: tau as diffeomorphism}, $\dim(\ker H) = k^2$ means that $\ker H$ equals the tangent space to the $\theta_A$-orbit of $(U_*,V_*)$, which however does not contain elements of the form $(U,0)$ or $(0,V)$ (this can be seen from~\eqref{eq: group action}). This allows to define a nonlinear SOR process in a neighborhood of such a critical point based on the implicit definitions~\eqref{eq: implicit definition S1 S2} of $\hat S_1$ and $\hat S_2$ even when $f$ is not strongly convex; cf.~\cite[Thm.~10.3.5]{Ortega1970}.
\end{remark}

To get a better intuition for the assumptions in the theorem, it is useful to write the Hessian as a bilinear form
\[
    \nabla^2 F(U_*,V_*)[h,h] = \langle \tau'(U_*,V_*)[h] ,\nabla^2 f(X_*) \cdot \tau'(U_*,V_*)[h] \rangle + \langle \nabla f(X_*), \tau''(U_*,V_*)[h,h] \rangle,
\]
where $h = (\delta U, \delta V)$. If $f$ is strictly convex, then the first term is nonnegative and equal to zero if and only if $\tau'(U_*,V_*)[h] = 0$, that is, if $h$ is in the tangent space to the $\theta_A$-orbit at $(U_*,V_*)$. Thus, an important situation in which the assumptions of the theorem are satisfied is when $\nabla f(X_*) = 0$, that is, when $X_*^{} = U_*^{} V_*^\top$ is a global minimum of $f$. In Section~\ref{sec: matrix completion} we conduct some numerical experiments for a matrix completion problem~\eqref{eq: matrix completion} admitting such a global minimum $X_*^{} = U_*^{}V_*^\top$ with $\nabla f(X_*) = 0$. In this application, however, $f$ is only convex, but not strictly convex. Then in order to satisfy the assumptions of the theorem at $X_*$ one would need that the tangent space $T_{X_*} \mathcal M_k$ (the image of $\tau'(U_*,V_*)$) does not intersect the null space of $\nabla^2 f(X_*) = P_\Omega$, but we will not investigate this condition in detail.

\subsection{Asymptotically optimal relaxation}

It is well-known that under certain assumptions the relaxation parameter $\omega$ in the linear SOR method can be optimized using a theorem of Young; see, e.g.,~\cite[Sec.~6.2]{Young71} or~\cite[Sec.~4.6.2]{Hackbusch2016}. This theory is usually presented for positive definite systems. However, for $2 \times 2$ block systems it is possible to adjust the arguments to the positive semidefinite case.

\begin{lemma}\label{lem: SOR method}
Let $H = D + E + E^\top \in \mathbb{R}^{p \times p}$ be a positive semidefinite $2 \times 2$ block matrix with positive definite block diagonal $D$ and such that $\frac{1}{\omega} D + E$ is invertible for any $0 < \omega < 2$. Assume $q = \dim(\ker H) < p/2$. 
Let $\sigma(I - D^{-1} H)$ denote the spectrum of $I - D^{-1}H$, then
\[
\beta \coloneqq \max \{ \abs{\mu} \colon \mu \in \sigma(I - D^{-1}H) \setminus \{\pm 1 \}  \}  < 1.
\]
The matrix $T_\omega = I - N_\omega^{-1} H$, where $N_\omega = \frac{1}{\omega} D + E$, induces a decomposition~\eqref{eq: splitting} into two invariant subspaces and the spectral radius $\rho_\omega$ of $T_\omega$ on $\mathcal W_\omega$ equals
\[
\rho_\omega = \begin{cases}
1 - \omega  + \frac{1}{2} \omega^2 \beta^2 + \omega \beta \sqrt{1 - \omega + \frac14 \omega^2 \beta^2}
 &\quad \text{if $0 < \omega \le \omega_{\text{\upshape opt}}$}, \\ \omega - 1 &\quad \text{if $\omega_{\text{\upshape opt}} \le \omega < 2$}, \end{cases}
\]
where
\begin{equation}\label{eq:omega_opt}
\omega_{\text{\upshape opt}} = \frac{2}{1 + \sqrt{1- \beta^2}} > 1.
\end{equation}
The value of $\rho_\omega$ is minimal for $\omega = \omega_{\text{\upshape opt}}$. It holds that $\beta^2 = \rho_1$ is the spectral radius for the standard AO method with $\omega = 1$ (on its invariant subspace $\mathcal W_1$). 
\end{lemma}

\begin{proof}
The decomposition~\eqref{eq: splitting} into invariant subspaces has already been verified (see Footnote~\ref{footnote 1}). We follow the arguments in the proof of~\cite[Thm.~4.27]{Hackbusch2016}. There it is shown that the eigenvalues $\mu$ of $I - D^{-1}H$ and $\lambda$ of $T_\omega$ are related via
\begin{equation}\label{eq: formula for lambda}
\lambda  = 1 - \omega + \frac{1}{2} \omega^2 \mu^2 \pm \omega \mu \sqrt{1 - \omega + \frac14 \omega^2 \mu^2}.
\end{equation}
Indeed, note that under the given assumptions $I - D^{-1}H = -D^{-1}(E + E^\top)$ is a two-cyclic matrix and has only real eigenvalues, of which the nonzero ones come in pairs $\pm \mu$. By~\eqref{eq: formula for lambda}, both $\mu$ and $-\mu$ create a same pair of eigenvalues $\lambda$. Eigenvalues $\mu = 1$ of $I - D^{-1}H$ must belong to eigenvectors in $\ker H$. Therefore, $\mu = 1$ and $\mu = -1$ both have multiplicity $q$. They yield eigenvalues $\lambda = 1$ and $\lambda = (1 - \omega)^2$ of $T_\omega$. Since eigenvectors of $T_\omega$ with $\lambda = 1$ must belong to $\ker H$, we conclude that the restriction of $T_\omega$ to the invariant subspace $\mathcal W_\omega$ has an eigenvalue $(1 - \omega)^2$ and its other eigenvalues are generated from formula~\eqref{eq: formula for lambda} with $\abs{\mu} \neq 1$. Since $2q < p$ such $\mu$ must exist. Rewriting the eigenvalue equation $(I - D^{-1}H)x = \mu x$ in the two ways
\[
Hx = (1 - \mu)Dx, \quad (2D - H)x = (1+\mu) D x,
\]
and using a special property of $2 \times 2$ block matrices that $H$ and $2D - H$ have the same eigenvalues, we obtain $\abs{\mu} \le 1$ since  both $H$ and $2D-H$ are positive semidefinite and $D$ is positive definite. This shows $\beta < 1$.

Consider eigenvalues $\mu$ of $I - D^{-1} H$ with $\abs{\mu} < 1$. If $\omega \ge \omega_{\text{\upshape opt}}$, then for such $\mu$ the expression under the square root in formula~\eqref{eq: formula for lambda} is always negative. Hence they generate pairs of conjugate complex eigenvalues $\lambda$, but one verifies that they all have the same modulus $\abs{\lambda} = \abs{1 - \omega}$, independent from $\abs{\mu}$. Clearly $\abs{1 - \omega} > (1 - \omega)^2$ so that the asserted formula for $\rho_\omega$ is proven for $\omega \ge \omega_{\text{\upshape opt}}$. When $0< \omega < \omega_{\text{\upshape opt}}$ the expression under the square root in~\eqref{eq: formula for lambda} may be negative or not. If it is negative, we have already seen that $\abs{\lambda} = \abs{1-\omega}$ is generated. If it is nonnegative, which in particular is the case for $\mu = \pm \beta$, the corresponding $\lambda$ with the larger absolute value is
\[
\lambda = 1 - \omega + \frac{1}{2} \omega^2 \mu^2 + \omega \abs{\mu} \sqrt{1 - \omega + \frac14 \omega^2 \mu^2}
\]
(since the sum before $\pm$ in~\eqref{eq: formula for lambda} then is nonnegative, too). This expression is maximized for $\mu = \pm\beta$ and also is then larger than $\abs{1 - \omega}$ on the interval $0<\omega < \omega_{\text{\upshape opt}}$. The statements of the lemma follow.
\end{proof}

\begin{remark}
In the setting of the lemma one always has $q = \dim(\ker H) \le p/2$, but (if $p$ is even) equality $q = 2p$ could in principle hold. It is then interesting to note that in this case $\rho_\omega = (1 - \omega)^2$, which is minimized for $\omega = 1$, yielding a superlinear convergence rate. However, this case is not relevant in the context of our work, where $p = km + kn$ with a rank $k < \min(m,n)$. In the following theorem we assume $q = k^2$ so that $q < p/2$ is satisfied. 
\end{remark}

Applying Lemma~\ref{lem: SOR method} in the context of Theorem~\ref{thm: local convergence theorem} immediately provides our main result on the asymptotically optimal choice of the shift $\omega$ for Algorithm~\ref{alg:als_qr}.

\begin{theorem}\label{thm: optimal shift for AO}
Let $f$ be strongly convex and $(U_*,V_*)$ be a critical point of $F$ in~\eqref{eq: factorized optimization problem} with $U_*, V_*$ having full column rank $k < \min(m,n)$. Assume that the Hessian $H = \nabla^2F (U_*,V_*)$ is positive semidefinite and $\dim (\ker H) = k^2$. Let $\rho_1 < 1$ be the asymptotic linear convergence rate of the standard AO method with $\omega = 1$. Fix $0 < \omega < 2$. Then for $X_0^{} = U_0^{} V_0^\top$ close enough (this may depend on $\omega$) to $X_*^{} = U_*^{} V_*^\top$ Algorithm~\ref{alg:als_qr} is well-defined and the iterates $X_\ell^{} = U_\ell^{} V_\ell^\top$ converge to $X_*$ at an asymptotic linear rate $\rho_\omega < 1$ given in Lemma~\ref{lem: SOR method} with $\beta^2 = \rho_1$. The optimal asymptotic rate is achieved for 
\[\omega_{\text{\upshape opt}} = \frac{2}{1 + \sqrt{1 - \beta^2}}.\]
\end{theorem}

In practice, the simplest approach for approximating $\omega_{\text{\upshape opt}}$ adaptively is by running the standard method with $\omega = 1$ and estimating $\beta^2 \approx \rho_1$ based on its numerically observed convergence rate. The efficiency of this approach will be illustrated in section~\ref{sec: numerical experiments}.

\section{Low-rank tensor problems}\label{sec: low-rank tensor problems}

Clearly, the nonlinear SOR method can be applied to functions with more than two block variables. In low-rank tensor optimization one frequently considers problems of the form
\begin{equation}\label{eq: multilinear optimization} 
\min F(U^1,\dots,U^D) = f(\tau(U^1,\dots,U^D)),
\end{equation}
where now $f$ is a smooth function on a tensor space $\R^{n_1 \times \dots \times n_d}$, and $\tau \colon \mathcal V_1 \times \dots \times \mathcal V_D \to \R^{n_1 \times \dots \times n_d}$ is a multilinear map parametrizing a low-rank tensor format. Such a problem is amenable to alternating optimization since the update for a single block variable $U^\mu$ is just an optimization problem for the function $f$, but on a linear subspace of $\R^{n_1 \times \dots \times n_d}$.

As an important example we mention optimization in the tensor train (TT) format~\cite{Oseledets2011}. Here $D = d$ and
\[
\tau \colon \mathcal V \coloneqq \R^{n_1 \times k_1} \times \R^{k_1 \times n_2 \times k_2} \times \dots \times \R^{k_{d-2} \times n_{d-1} \times k_{d-1}} \times \R^{k_{d-1} \times n_d} \to \R^{n_1 \times \dots \times n_d}
\]
is defined via
\begin{equation}\label{eq: TT map}
X = \tau(U^1, \dots, U^d) \quad \Leftrightarrow \quad X(i_1,\dots,i_d) = U^1(i_1,\colon) U^2(\colon,i_2,\colon) \cdots  U^d(\colon,i_d),
\end{equation}
which are matrix products of corresponding slices in the so called TT cores $U^\mu \in \R^{k_{\mu-1} \times n_\mu \times k_\mu}$ (one fixes $k_0 = k_1 = 1$). The minimial possible values $(k_1,\dots,k_{d-1})$, which determine the sizes of the TT cores, such that such a decomposition is possible are called the TT-ranks of tensor $X$. Alternating optimization methods form the basis for the majority of computational methods in the tensor train format~\cite{Holtz2012}.

An AO method with relaxation for~\eqref{eq: multilinear optimization} takes the form
\begin{equation}\label{eq: AO for multilinear}
\begin{aligned}
    U^1_{\ell + 1} &= (1-\omega) U^1_\ell + \omega \hat S_1(U^2_\ell,\dots,U^d_\ell), \\
    &\vdots\\
    U^{\mu}_{\ell+1} &= (1-\omega) U^\mu_\ell + \omega \hat S_\mu(U^1_{\ell+1},\dots,U^{\mu-1}_{\ell+1},U^{\mu+1}_{\ell},\dots,U^d_\ell), \\
    &\vdots \\
    U^d_{\ell + 1} &= (1-\omega) U^d_\ell + \omega \hat S_d(U^1_{\ell+1},\dots,U^{d-1}_{\ell+1}),
\end{aligned}
\end{equation}
where the $\hat S_\mu$ return minimizers (or critical points) of the restricted functions $U^\mu \mapsto F(\ldots,U^\mu,\ldots)$ with the other block variables being fixed. In this form the method has been proposed for the tensor train format in~\cite{Grasedyck2015}. Under suitable assumptions such an iteration defines a smooth map $S_\omega$ from an open subset of $\mathcal V_1 \times \dots \times \mathcal V_D$ to $\R^{n_1 \times \dots \times n_d}$ for which a similar fixed point analysis as in the matrix case can be conducted. In the following we sketch this for the tensor train format, but the ideas can be applied to general tree tensor network formats such as the Tucker or hierarchical Tucker format. We will make use of several well-known properties of the tensor train format, in particular the quotient manifold structure of tensors of fixed TT-rank and the orbital invariance of AO methods. Most of the related details can be found in~\cite{Rohwedder2013} and~\cite{Uschmajew2013}. 

For the tensor train format~\eqref{eq: TT map} we assume that $\mathbf k = (k_1,\dots,k_{d-1})$ is chosen such that tensors of TT-rank $\mathbf k$ exist. Then in fact on a dense and open subset $\mathcal V'$ of $\mathcal V$ the map $\tau$ maps to tensors of fixed TT-rank $\mathbf k$. For convenience we will use the notation $\mathbf U = (U^1,\dots,U^d)$ for the elements in $\mathcal V$. As in the matrix case, $\tau$ in~\eqref{eq: TT map} is invariant under a group action, namely,
\begin{equation*}\label{eq: group action TT}
\mathcal G = \mathrm{GL}(k_1) \times \dots \times \mathrm{GL}(k_{d-1}) \ni \mathbf A = (A_1,\dots,A_{d-1}) \mapsto \theta_{\mathbf A} \cdot \mathbf U,
\end{equation*}
which inserts the product $A_{\mu}^{} A_{\mu}^{-1}$ between the matrices $U^\mu(:,i_\mu,:)$ and $U^{\mu+1}(:,i_{\mu+1},:)$ in~\eqref{eq: TT map}, that is, the slices of the TT cores are transformed according to
\begin{equation}\label{eq: group action slicewise}
U^\mu(:,i_\mu,:) \to A_{\mu-1}^{-1} U^\mu(:,i_\mu,:) A_\mu^{}
\end{equation}
(here $A_0 = A_d = 1$). The corresponding restriction of $\tau$ to the quotient manifold $\mathcal V' / \mathcal G$ is a diffeomorphism onto the set $\mathcal M_{\mathbf k}$ of tensors of fixed TT-rank $\mathbf k$, which is an embedded submanifold of $\R^{n_1 \times \dots \times n_d}$ of dimension $\dim(\mathcal M_{\mathbf k}) = \dim(\mathcal V) - \dim(\mathcal G)$. Notably, let $\mathbf U \in \mathcal V'$, then $\tau'(\mathbf U) = 0$ on the tangent space of the orbit $\theta_{\mathbf A} \cdot \mathbf U$ at $\mathbf U$. On any complementary subspace $\mathcal W$ to that tangent space, $\tau'(\mathbf U)$ is a bijection from $\mathcal W$ to the tangent space of $\mathcal M_{\mathbf k}$ at $\tau(\mathbf U)$.

Assume again that $f$ is smooth and strongly convex. Then any critical point $\mathbf U_*$ of the function $F = f \circ \tau$ that lies in $\mathcal V'$ is a fixed point of the iteration~\eqref{eq: AO for multilinear} since the restricted linear maps $U^\mu \mapsto \tau(U^1_*,\dots,U^\mu,\dots,U_*^d)$ are injective so that the corresponding restriction $U^\mu \mapsto F(U^1_*,\dots,U^\mu,\dots,U_*^d)$ is strongly convex. Moreover, the whole process is well-defined in some neighborhood of (the orbit of) $\mathbf U_*$ where it can be written as
\[
\mathbf U_{\ell+1} = S_\omega (\mathbf U_\ell)
\]
with a smooth map $S_\omega$. A key observation to make is that the maps $\hat S_1,\dots, \hat S_d$ in~\eqref{eq: AO for multilinear} that realize the updates of single TT cores exhibit an analogous compatibility with the group action as in~\eqref{eq: S1 S2 invariance} for the matrix case, namely
\[
\hat S_\mu(\theta_{\mathbf{A}} \cdot \mathbf U) = A_{\mu-1}^{-1} \hat S_\mu(\mathbf U) A_\mu^{},
\]
where the matrix product is understood slice-wise as in~\eqref{eq: group action slicewise} (and we slightly abused notation since $\hat S_\mu$ does not depend on $U^\mu$). It entails a corresponding invariance
\begin{equation}\label{eq: group invariance TT}
S_\omega (\theta_{\mathbf A} \cdot \mathbf U) = \theta_{\mathbf A} \cdot S_\omega (\mathbf U)
\end{equation}
of a full update loop, in analogy to~\eqref{eq: invariance for Somega}. This allows us to regard~\eqref{eq: AO for multilinear} as a well-defined iteration
\[
X_{\ell+1} = \mathbf S_\omega(X_\ell)
\]
on the manifold $\mathcal M_{\mathbf k}$, at least locally in a neighborhood of $\tau(\mathbf U_*)$. From a practical viewpoint, the invariance~\eqref{eq: group invariance TT} admits to change the tensor train representation in every substep of~\eqref{eq: AO for multilinear} in order to make the restricted linear maps $U^\mu \mapsto \tau(\dots,U^\mu,\dots,)$ orthogonal and improve numerical stability. We refer to~\cite{Holtz2012,Rohwedder2013} for details on orthogonalization of substeps.

Based on these similarities to the matrix case, one can proceed in almost the same way as in section~\ref{sec: low-rank AO}. Let $\mathbf U_* \in \mathcal V'$ be a critical point of $F$, that is, $\nabla F(\mathbf U_*) = 0$. Due to the orbital invariance of $F$, the Hessian $H = \nabla^2 F(\mathbf U_*)$ has a kernel of dimension at least $\dim(\mathcal G)$ since it contains the tangent space to the orbit at $\mathbf U_*$. However, in the block decomposition
\begin{equation}\label{eq: decomposition}
H = D + E + E^\top
\end{equation}
into a block diagonal part $D$ (corresponding to the block variables $U^1,\dots,U^d$) and lower block triangular part $E$, the block matrix $D$ is positive definite since $f$ is strongly convex.\footnote{As in Remark~\ref{rem: D positive definite}, this also follows from the assumptions that $H$ is positive semidefinite and $\dim(\ker H) = \dim (\mathcal G)$, since the tangent space to the orbit at $\mathbf U_* \in \mathcal V'$ does not contain elements of the form $(0,\dots,0,U^\mu,0,\dots,0)$.} The derivative of $S_\omega$ then again takes the form of an SOR error iteration matrix
\[
T_\omega = I - N_\omega^{-1} H, \quad N_\omega = \frac{1}{\omega}D + E,
\]
similar to~\eqref{eq: derivative in parameter space}; see, e.g.,~\cite[Thm.~10.3.4 \& 10.3.5]{Ortega1970} for the derivation. For $0 < \omega < 2$ a decomposition $\mathcal V = \ker H \oplus \mathcal W_\omega$ as in~\eqref{eq: splitting} applies and $T_\omega$ is a contraction on the invariant subspace $\mathcal W_\omega$ if $H$ is positive semidefinite. Using the same proof as for Theorem~\ref{thm: local convergence theorem} we obtain the analogous local convergence result for tensor train optimization. Recall that we assume that~$\mathbf k$ is properly chosen so that $\tau$ maps the open and dense subset $\mathcal V'$ to the manifold $\mathcal M_{\mathbf k}$.

\begin{theorem}\label{thm: local convergence TT}
Let $\mathbf U_* \in \mathcal V'$ be a critical point of function $F$ in~\eqref{eq: multilinear optimization} where $f$ is strongly convex. Assume that the Hessian $H = \nabla^2F (\mathbf U_*)$ is positive semidefinite and $\dim (\ker H) = \dim(\mathcal G) = k_1^2 + \dots + k_{d-1}^2$. Fix $0 < \omega < 2$. Then for $X_0 = \tau(\mathbf U_0)$ close enough (this may depend on $\omega$) to $X_* = \tau(\mathbf U_*)$ the iteration~\eqref{eq: AO for multilinear} is well-defined and the iterates $X_\ell = \tau(\mathbf U_\ell)$ converge to $X_*$ at an asymptotic linear rate
\[
\rho_\omega = \limsup_{\ell \to \infty} \| X_\ell - X_* \|^{1/\ell} < 1,
\]
where $\rho_\omega$ is the spectral radius of $T_\omega$ on $\mathcal W_\omega$.
\end{theorem}

While so far everything looks conceptually almost identical to the matrix case, a major difference arises when proceeding to determine the optimal shift parameter $\omega$. It is not clear whether Lemma~\ref{lem: SOR method} can be generalized. This is in fact already an issue with the linear SOR method with more than two blocks for positive definite systems, since certain conditions on the decomposition~\eqref{eq: decomposition} of $H$ are required in order to derive a formula like~\eqref{eq: formula for lambda} for the eigenvalues of $T_\omega$; cf.~\cite[Sec.~4.6]{Hackbusch2016}. In the matrix case, the fact that $E + E^\top$ is two-cyclic makes this possible but for more than two block variables assuming such conditions on $E$ does not appear very reasonable, especially when taking into account that the critical point $(U_*,V_*)$ and hence its Hessian are not given a priori. Moreover, even if the formula~\eqref{eq: formula for lambda} would apply, one would need that the eigenvalues of the Jacobi error iteration matrix $I - D^{-1}H$ have absolute value at most one, but for a block decomposition~\eqref{eq: decomposition} with more than two blocks this does not follow from the positive definiteness of $D$ alone. Thus, for the tensor train format the estimation of an optimal parameter $\omega_{\text{\upshape opt}}$ from formula~\eqref{eq:omega_opt} remains a heuristic.

\section{Numerical experiments}\label{sec: numerical experiments}

In the last section we present some numerical experiments to illustrate the benefit of overrelaxation in low-rank optimization.

\subsection{Matrix completion problem}\label{sec: matrix completion}

\begin{figure}
\centering
\begin{subfigure}{.5\textwidth}
  \includegraphics[width=\linewidth]{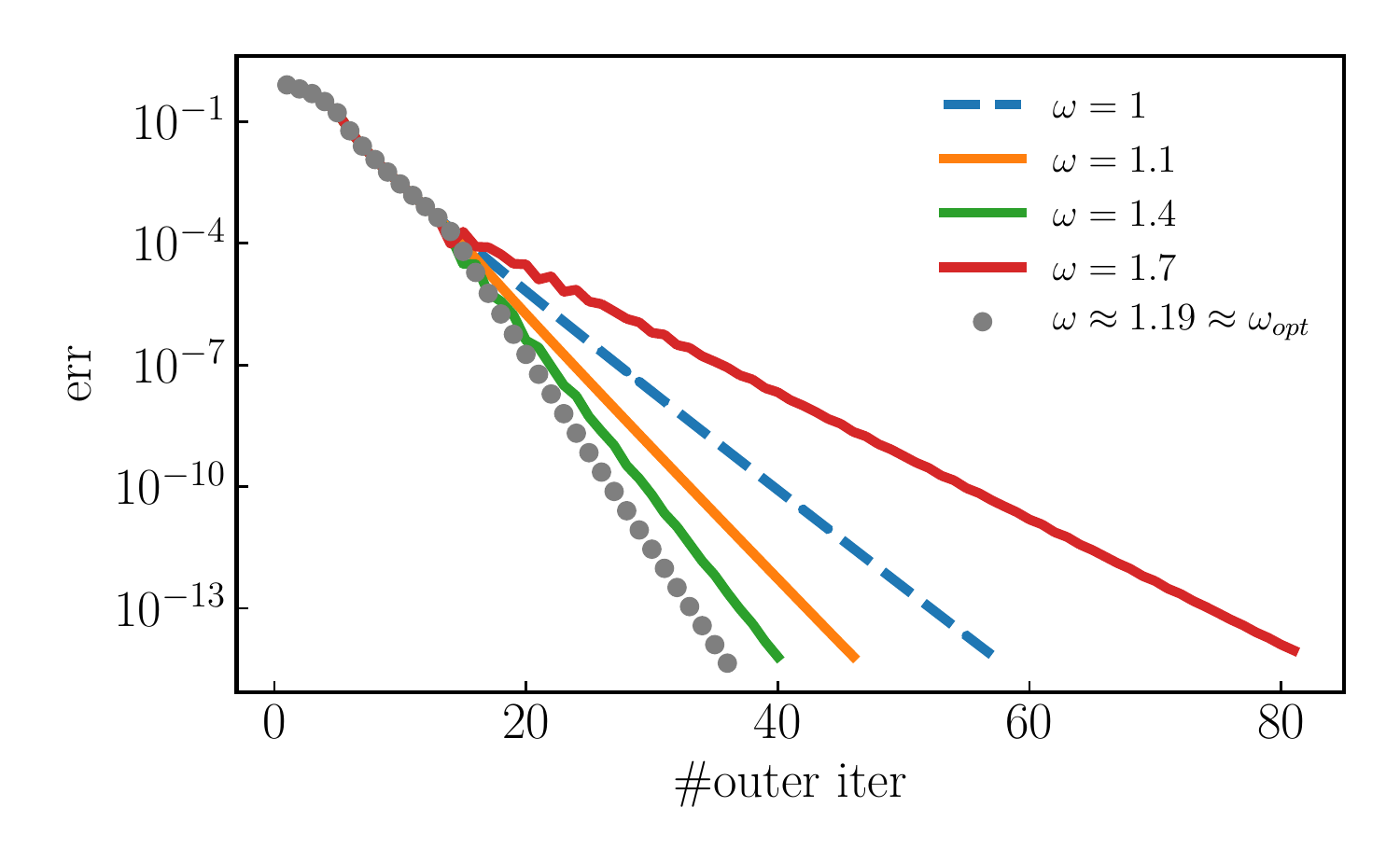}
\end{subfigure}%
\begin{subfigure}{.5\textwidth}
  \includegraphics[width=\linewidth]{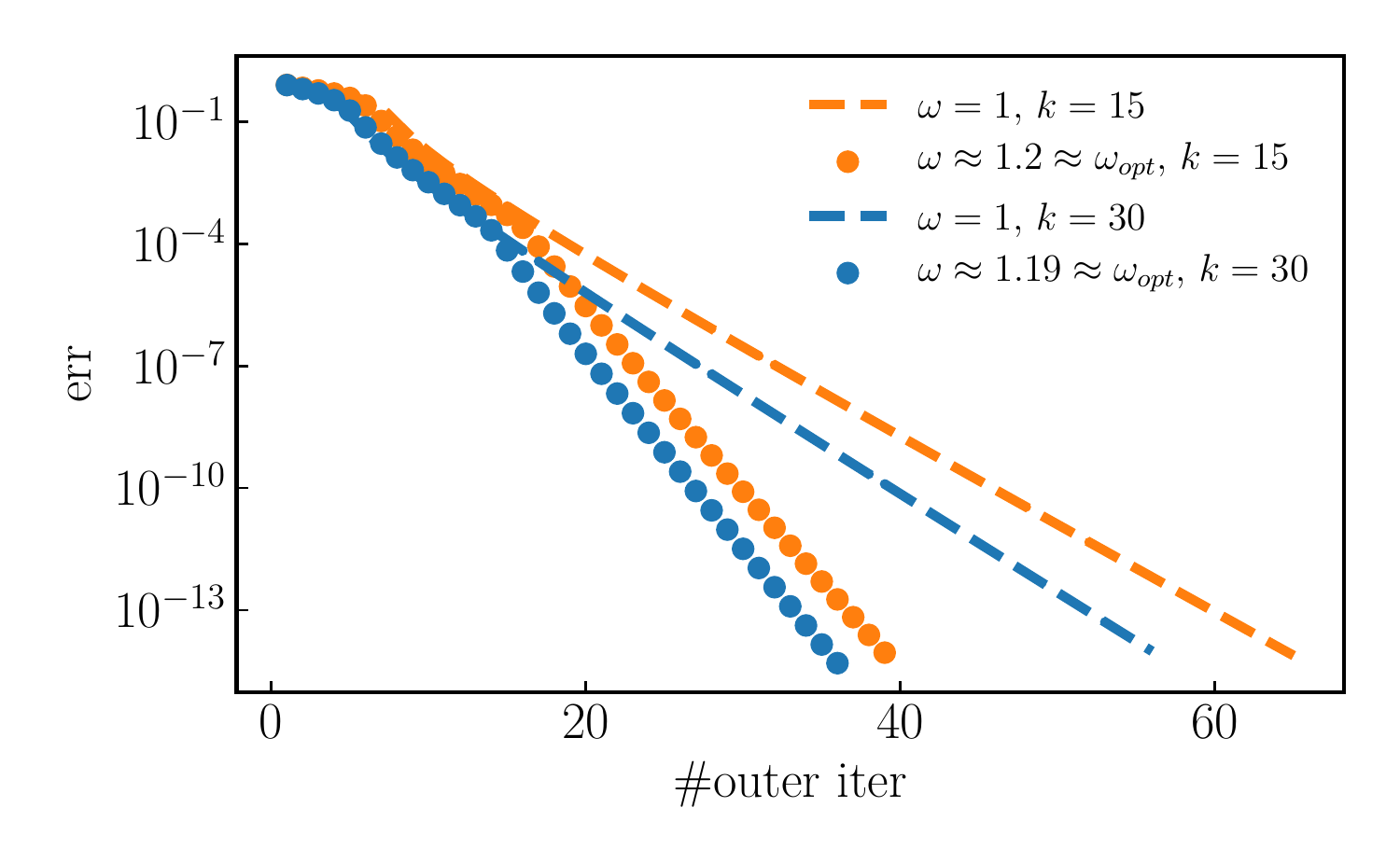}
\end{subfigure}
\caption{Relative residuals~\eqref{eq:rel_err} of Algorithm~\ref{alg:als_qr} for the completion problem~\eqref{eq: matrix completion} with respect to the number of outer iterations using various shift parameters $\omega$ for rank $k=30$ (left) and for rank values $k=15, 30$ (right). The parameter $\omega=1$ corresponds to the standard ALS method, while parameters $\omega\in(1, 2)$ represent the version of the iteration with overrelaxation. The $\omega \approx \omega_{\mathrm{opt}}$ case corresponds to the choice~\eqref{eq:omega_opt} with $\beta^2$ estimated using~\eqref{eq:param}.
}
\label{fig:completion_random}
\end{figure}

First, we apply the proposed AO overrelaxation scheme in Algorithm~\ref{alg:als_qr} to the following nonconvex formulation of a low-rank matrix completion problem:
\begin{equation}\label{eq: matrix completion}
  \min_{U\in\mathbb{R}^{m\times k}, V\in\mathbb{R}^{n\times k}} F(U,V) = \frac{1}{2}\left\|P_{\Omega} (A - UV^\top) \right\|_F^2.
  \end{equation}
Here $\Omega$ is a given set of index pairs, and the linear operator $P_{\Omega}\colon \mathbb{R}^{m\times n} \to \mathbb{R}^{m\times n}$ is defined as 
\[
    \left(P_{\Omega}(X) \right)_{ij} = 
    \begin{cases}
        x_{ij}, & (i,j)\in\Omega, \\
        0, & \text{ otherwise}.
    \end{cases}
\]
In our experiments the set $\Omega$ consists of randomly generated index pairs.
We set $m=n$ and choose $A$ to be a random rank-$k$ matrix, i.e., $A = U_*^{} V_*^\top$, where $U_*,V_*\in\mathbb{R}^{n\times r}$ are random matrices with each element sampled from a standard Gaussian distribution. The number of sampled entries $|\Omega|$ is defined by an oversampling parameter $\texttt{OS}\geq 1$,
\[
    \abs{\Omega} = \texttt{OS} \cdot (2nk - k^2),
\]
since we want $\abs{\Omega}$ to be larger than $2nk - k^2 = \dim(\mathcal M_k)$, which is the number of essential degrees of freedom for an $n \times n$ matrix of rank $k$. In the experiments~$\texttt{OS} = 3$.

In Figure~\ref{fig:completion_random} we present the convergence plots for an experiment with $n=2000$ for several choices of $\omega$ and different ranks $k$. We report the relative residuals
\begin{equation}\label{eq:rel_err}
    \texttt{err}_\ell = \frac{\left\|P_{\Omega} (A - U_\ell^{} V_\ell^\top) \right\|_F}{\left\|P_{\Omega} (A) \right\|_F},
\end{equation}
where the sequence $(U_\ell, V_\ell)$ is generated by Algorithm~\ref{alg:als_qr} with a shift parameter $\omega$.
The only difference with Algorithm~\ref{alg:als_qr} is that we always start with $\omega = 1$ (standard ALS) and only turn on the shift after the convergence has stabilized, in this experiment usually after 12 iterations. The optimal shift $\omega_\mathrm{opt}$ from \eqref{eq:omega_opt} depends on the convergence rate $\beta^2 = \rho_1$ of the standard AO method, which is estimated while running the iteration with $\omega=1$ using
\begin{equation}\label{eq:param}
    \beta^2 \approx \sqrt{\frac{\texttt{err}_{\ell+2}}{\texttt{err}_{\ell}}}.
\end{equation}

As expected, using overrelaxation accelerates the convergence of the ALS method if $\omega$ is chosen properly.
We note that the additional computations arising from the overrelaxation scheme come in asymptotically negligible cost as compared with the basic ALS.
In turn, the proposed approach leads to a significant reduction of the total number of iterations for achieving a high accuracy.

\subsection{Low-rank solution of the Lyapunov equation}

\begin{figure}
\centering
\includegraphics[width=0.6\textwidth]{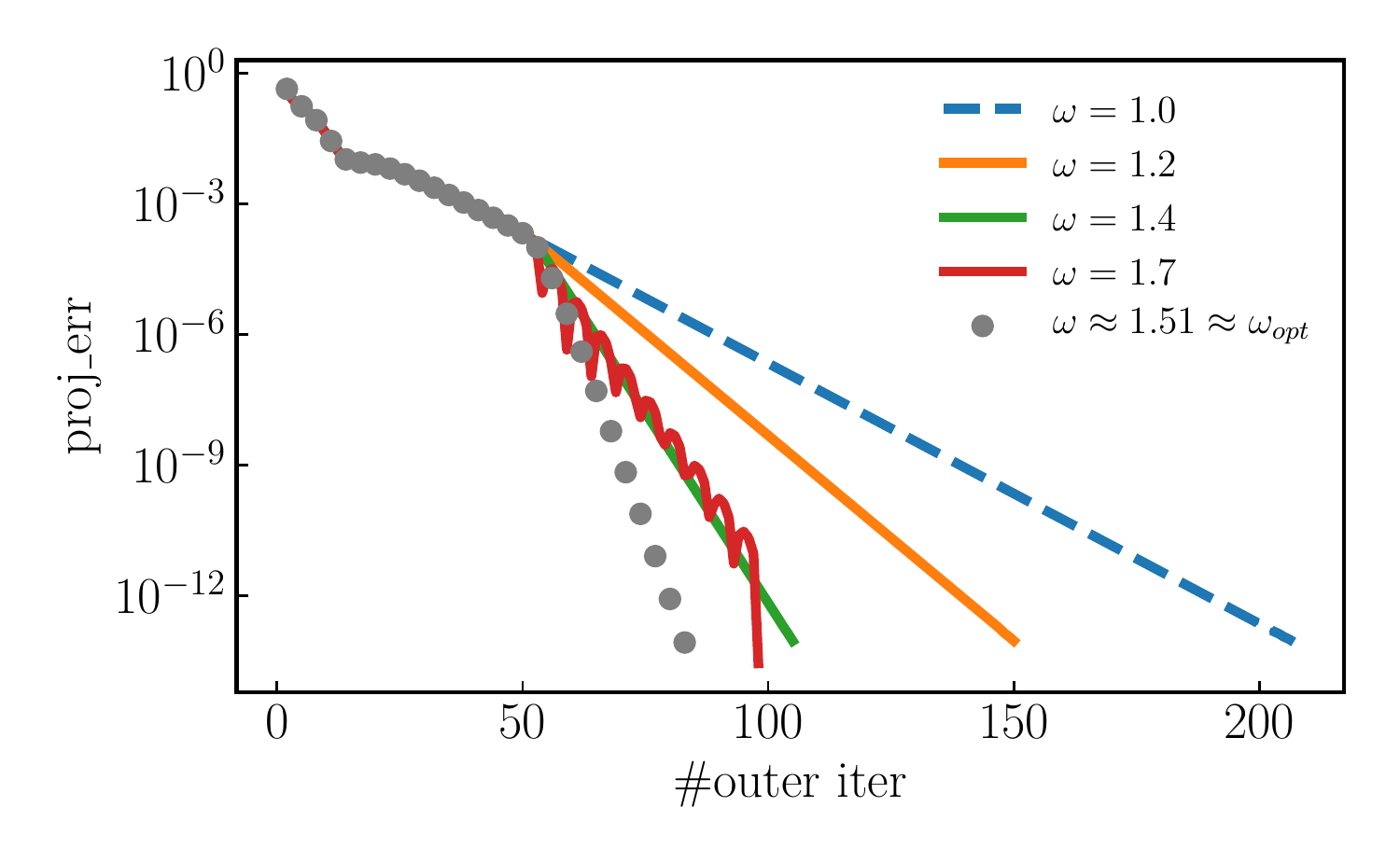}
\caption{
Relative residuals~\eqref{eq:proj_err} of Algorithm~\ref{alg:als_qr} for the low-rank solution of a Lyapunov equation~\eqref{eq:lyapunov} with respect to the number of outer iterations using various shift parameters $\omega$. Here $k=2$. The parameter $\omega=1$ corresponds to the standard ALS method, while parameters $\omega\in(1, 2)$ represent the version of the iteration with overrelaxation. The $\omega \approx \omega_{\mathrm{opt}}$ case corresponds to the choice~\eqref{eq:omega_opt} with $\beta^2$ estimated using~\eqref{eq:param_proj}.
}
\label{fig:system2d}
\end{figure}

Consider the Lyapunov equation
\begin{equation}\label{eq:lyapunov}
    AX + XA^\top = B, \quad A, B\in\mathbb{R}^{n\times n},
\end{equation}
where $X\in\mathbb{R}^{n\times n}$ is a matrix to be found.
In case of a symmetric positive definite matrix $A$, equation \eqref{eq:lyapunov} represents the optimality condition for the strongly convex optimization problem
\[
   \min_{X \in \R^{n \times n}} f(X) = \frac{1}{2}\left<AX + XA^\top, X\right>_F - \left< B, X \right>_F.
\]
A rank-$k$ approximation to the solution is, therefore, obtained by solving the problem
\[
    \min_{U, V\in\mathbb{R}^{n\times k}} F(U,V) = f\left(UV^\top\right)
\]
instead, which is of the form~\eqref{eq: factorized optimization problem}. For this we employ the proposed overrelaxation algorithm.

In the experiment, we choose $A = (n+1)^2\, \mathrm{tridiag(-1, 2, -1)}$, set $n=256$ and generate the right-hand side $B = AX_* + X_* A^\top$ from a specified solution $X_*$. 
Specifically, we choose $k=2$ and generate the second and the third singular values of $X_*$ such that their ratio equals to $0.99$, which is similar to experiments conducted in~\cite{Oseledets2018}. There it has been numerically observed that such a large ratio at the target singular value can lead to slow convergence of the standard ALS method.

Due to the fact that $X_*$ cannot be approximated with high accuracy using $k=2$, the function values $f(X)$ will not converge to zero and hence cannot be taken as an appropriate error measure. Instead, we compute the values
\begin{equation}\label{eq:proj_err}
    \texttt{proj\_err}_\ell = \frac{\|\mathbf P_{X_\ell} \nabla f(X_\ell) \|_F }{\|\mathbf P_{X_\ell}(B)\|_F} = \frac{\|\mathbf P_{X_\ell}\left(A X_\ell+ X_\ell A^\top - B\right)\|_F}{\|\mathbf P_{X_\ell}(B)\|_F}, \quad X_\ell^{} = U_\ell^{} V_\ell^\top,
\end{equation}
where $\mathbf P_{X_\ell}$ denotes the orthogonal projection operator to the tangent space of the manifold $\mathcal M_k$ of fixed rank-$k$ matrices at $X_\ell$; see, e.g.,~\cite{Vandereycken2013}. This reflects the fact that the method can be regarded as a minimization method on that manifold. Similarly to~\eqref{eq:param}, we can then use these values for approximating the optimal shift parameter $\omega_\mathrm{opt}$ using \eqref{eq:omega_opt} with $\beta^2$ estimated from
\begin{equation}\label{eq:param_proj}
    \beta^2 \approx \sqrt{\frac{\texttt{proj\_err}_{\ell+2}}{\texttt{proj\_err}_{\ell}}}.
\end{equation}

In Figure~\ref{fig:system2d}, we plot the values of $\texttt{proj\_err}_\ell$ against the number of outer iterations~$\ell$ for several values of shifts~$\omega$, including the basic ALS and the approximated optimal shift. In all cases the shift is activated after 50 iterations. All considered shifts lead to convergence improvement with the shift that approximates the optimal one being the best.

\subsection{Linear systems in the quantized tensor train format}

\begin{figure}
\centering
\includegraphics[width=0.6\textwidth]{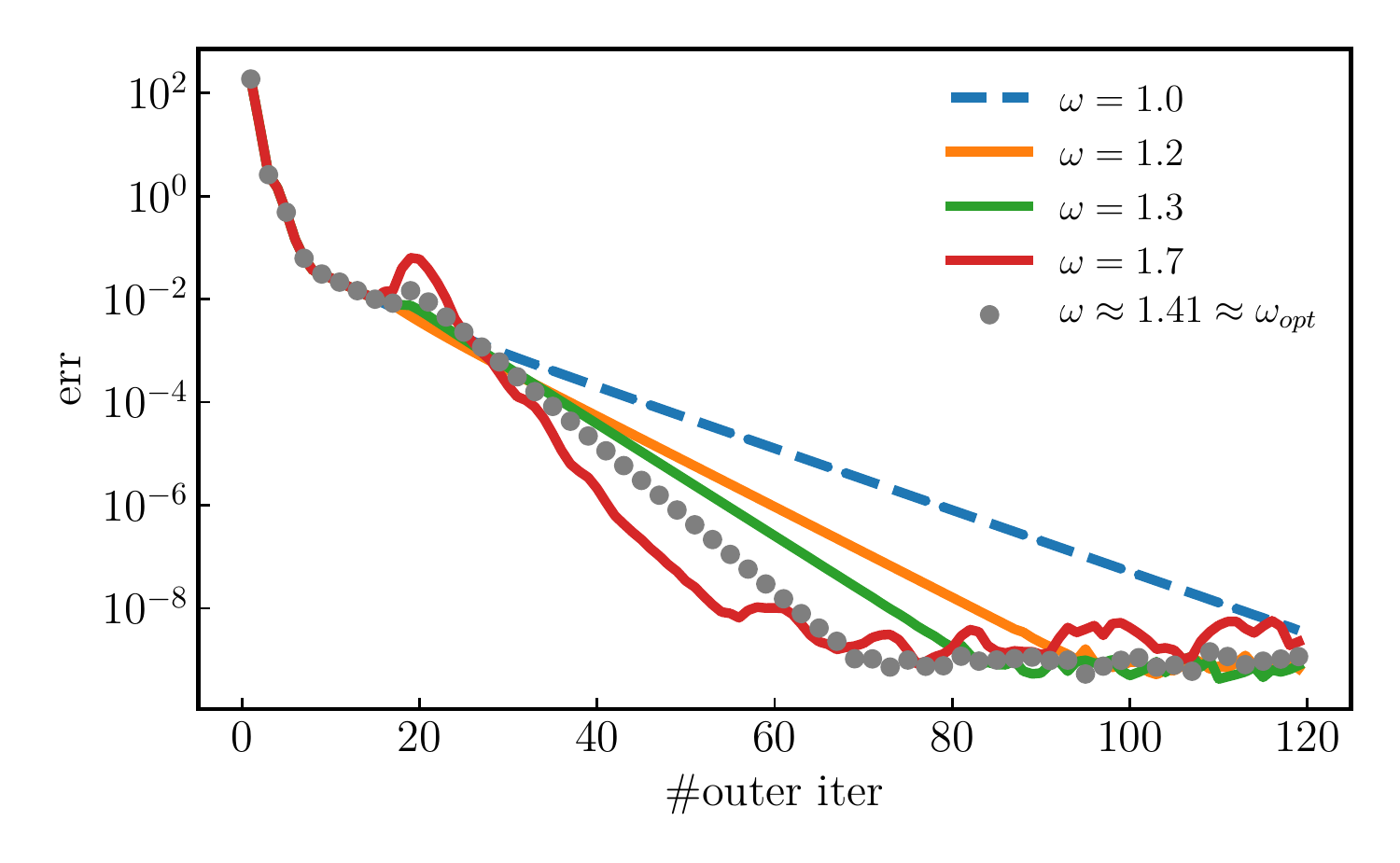}
\caption{Maximal (within one ALS sweep) relative local residual with respect to the number of outer iterations of the QTT ALS method for solving a Lyapunov equation~\eqref{eq:lyapunov} using various shifts $\omega$. The QTT ranks are $(4,\dots,4)$. The parameter $\omega=1$ corresponds to the standard QTT ALS method, while parameters $\omega\in(1, 2)$ represent the iteration with overrelaxation. The $\omega \approx \omega_{\mathrm{opt}}$ case corresponds to~\eqref{eq:omega_opt} with $\beta^2$ estimated from the standard ALS method, even though for tensor problems there is no theoretical guarantee that this choice is actually close to optimal. Shifts are activated after 15 iterations.
}
\label{fig:system_qtt}
\end{figure}

Finally, we test our approach for solving linear systems in the tensor train format. In particular, we apply the so-called quantized tensor train (QTT) format~\cite{Khoromskij2011,Oseledets2010} to solve the equation~\eqref{eq:lyapunov} by fixing $n=2^d$, $d=12$, and by representing $X\in\mathbb{R}^{2^d\times 2^d}$ as order-$2d$ tensors in $\mathbb{R}^{2\times \dots \times 2}$ using reshape in the lexicographical order.
These tensors are then further restricted to the tensor train format with the TT-rank equal to $(4,4,\dots,4)$ (this choice of ranks led to a much slower convergence of the ALS method as compared to other rank values). The right-hand side $B$ was selected to be a matrix of all ones, which trivially admits a QTT representation with all TT-ranks equal to one. Note that all computations were performed directly in the tensor train format, i.e., no full tensors were formed.

As an error measure $\texttt{err}_\ell$ we take the maximum relative norm of all local residuals within one sweep of the standard ALS~\cite{Oseledets2012}. Based on this error we estimate $\beta^2$ and use the same formula~\eqref{eq:param} for $\omega_{\text{\upshape opt}}$, but as noted in section~\ref{sec: low-rank tensor problems} there is no theoretical guarantee that this formula provides the optimal shift parameter in the tensor case. Nevertheless, the results from Figure~\ref{fig:system_qtt} suggest that this choice leads to nearly the fastest convergence among the considered choices of shifts.

\small
\bibliographystyle{plain}
\bibliography{references}

\begin{thebibliography}{10}

\bibitem{Grasedyck2015}
L.~Grasedyck, M.~Kluge, and S.~Kr\"{a}mer.
\newblock Variants of alternating least squares tensor completion in the tensor
  train format.
\newblock {\em SIAM J. Sci. Comput.}, 37(5):A2424--A2450, 2015.

\bibitem{Hackbusch2016}
W.~Hackbusch.
\newblock {\em Iterative solution of large sparse systems of equations}.
\newblock Springer, [Cham], second edition, 2016.

\bibitem{Hageman1975}
L.~A. Hageman and T.~A. Porsching.
\newblock Aspects of nonlinear block successive overrelaxation.
\newblock {\em SIAM J. Numer. Anal.}, 12:316--335, 1975.

\bibitem{Holtz2012}
S.~Holtz, T.~Rohwedder, and R.~Schneider.
\newblock The alternating linear scheme for tensor optimization in the tensor
  train format.
\newblock {\em SIAM J. Sci. Comput.}, 34(2):A683--A713, 2012.

\bibitem{Keller1965}
H.~B. Keller.
\newblock On the solution of singular and semidefinite linear systems by
  iteration.
\newblock {\em J. Soc. Indust. Appl. Math. Ser. B Numer. Anal.}, 2:281--290,
  1965.

\bibitem{Khoromskij2011}
B.~N. Khoromskij.
\newblock {$O(d\log N)$}-quantics approximation of {$N$}-{$d$} tensors in
  high-dimensional numerical modeling.
\newblock {\em Constr. Approx.}, 34(2):257--280, 2011.

\bibitem{Lee2003}
J.~M. Lee.
\newblock {\em Introduction to smooth manifolds}.
\newblock Springer-Verlag, New York, 2003.

\bibitem{Ortega1970}
J.~M. Ortega and W.~C. Rheinboldt.
\newblock {\em Iterative solution of nonlinear equations in several variables}.
\newblock Academic Press, New York-London, 1970.

\bibitem{Ortega1966}
J.~M. Ortega and M.~L. Rockoff.
\newblock Nonlinear difference equations and {G}auss-{S}eidel type iterative
  methods.
\newblock {\em SIAM J. Numer. Anal.}, 3:497--513, 1966.

\bibitem{Oseledets2010}
I.~V. Oseledets.
\newblock Approximation of {$2^d\times 2^d$} matrices using tensor
  decomposition.
\newblock {\em SIAM J. Matrix Anal. Appl.}, 31(4):2130--2145, 2009/10.

\bibitem{Oseledets2011}
I.~V. Oseledets.
\newblock Tensor-train decomposition.
\newblock {\em SIAM J. Sci. Comput.}, 33(5):2295--2317, 2011.

\bibitem{Oseledets2012}
I.~V. Oseledets and S.~V. Dolgov.
\newblock Solution of linear systems and matrix inversion in the {TT}-format.
\newblock {\em SIAM J. Sci. Comput.}, 34(5):A2718--A2739, 2012.

\bibitem{Oseledets2018}
I.~V. Oseledets, M.~V. Rakhuba, and A.~Uschmajew.
\newblock Alternating least squares as moving subspace correction.
\newblock {\em SIAM J. Numer. Anal.}, 56(6):3459--3479, 2018.

\bibitem{Rohwedder2013}
T.~Rohwedder and A.~Uschmajew.
\newblock On local convergence of alternating schemes for optimization of
  convex problems in the tensor train format.
\newblock {\em SIAM J. Numer. Anal.}, 51(2):1134--1162, 2013.

\bibitem{Schechter1962}
S.~Schechter.
\newblock Iteration methods for nonlinear problems.
\newblock {\em Trans. Amer. Math. Soc.}, 104:179--189, 1962.

\bibitem{Uschmajew2012}
A.~Uschmajew.
\newblock Local convergence of the alternating least squares algorithm for
  canonical tensor approximation.
\newblock {\em SIAM J. Matrix Anal. Appl.}, 33(2):639--652, 2012.

\bibitem{Uschmajew2013}
A.~Uschmajew and B.~Vandereycken.
\newblock The geometry of algorithms using hierarchical tensors.
\newblock {\em Linear Algebra Appl.}, 439(1):133--166, 2013.

\bibitem{Vandereycken2013}
B.~Vandereycken.
\newblock Low-rank matrix completion by {R}iemannian optimization.
\newblock {\em SIAM J. Optim.}, 23(2):1214--1236, 2013.

\bibitem{Wang2019}
J.~Wang, Y.-P. Wang, Z.~Xu, and C.-L. Wang.
\newblock Accelerated low rank matrix approximate algorithms for matrix
  completion.
\newblock {\em Comput. Math. Appl.}, 77(2):334--341, 2019.

\bibitem{Weissinger1953}
J.~Weissinger.
\newblock Verallgemeinerungen des {S}eidelschen {I}terationsverfahrens.
\newblock {\em Z. Angew. Math. Mech.}, 33:155--163, 1953.
\newblock (German).

\bibitem{Wen2012}
Z.~Wen, W.~Yin, and Y.~Zhang.
\newblock Solving a low-rank factorization model for matrix completion by a
  nonlinear successive over-relaxation algorithm.
\newblock {\em Math. Program. Comput.}, 4(4):333--361, 2012.

\bibitem{Young71}
D.~M. Young.
\newblock {\em Iterative solution of large linear systems}.
\newblock Academic Press, New York-London, 1971.

\end{thebibliography}

\end{document}